\documentclass{amsart}
\usepackage{amssymb}
\usepackage{graphicx}

\DeclareSymbolFont{AMSb}{U}{msb}{m}{n}
\DeclareSymbolFontAlphabet{\Bb}{AMSb}

\usepackage{bbm}

\usepackage{tikz}

\usetikzlibrary{intersections,calc,positioning}
\usetikzlibrary{shapes.geometric,fit}
\usetikzlibrary{decorations,arrows,decorations.pathmorphing,backgrounds,positioning,fit,petri,shadows}
\usetikzlibrary{3d}
\usepackage{yhmath}
\usepackage{mathdots}
\usepackage{MnSymbol}

\newtheorem{theorem}{Theorem}[section]

\newtheorem{proposition}[theorem]{Proposition}
\theoremstyle{remark}

\newtheorem{definition}[theorem]{Definition}
\numberwithin{equation}{section}
\begin{document}

\title[Combinatorial relations among relations]
{Combinatorial relations among relations of $C_{2}\sp{(1)}$-standard modules for higher levels }
\author{Tomislav \v Siki\' c }
\address{Tomislav \v{S}iki\'{c}, University of Zagreb, Faculty of Electrical Engineering and Com- \mbox{\hskip 8.5em} puting, Unska 3, 10000 Zagreb,  Croatia}
\email{tomislav.sikic@fer.hr}


\begin{abstract}
For an affine Lie algebra $\hat{\mathfrak g}$ the coefficients of certain vertex operators which annihilate level $k$ standard $\hat{\mathfrak g}$-modules are the defining relations for level $k$ standard modules.  In the paper \cite{PS3} combinatorial structure of the leading terms of mentioned relations for level $k=2$ standard  $\hat{\mathfrak g}$-modules and affine Lie algebras of type $C_{n}\sp{(1)}$ is given. The main result of this paper is a construction of combinatorially parameterized relations among relations of $C_{2}\sp{(1)}$-standard modules for all levels.
\end{abstract}
\maketitle
\def\sq{{\lower.3ex\vbox{\hrule\hbox{\vrule height1.2ex depth1.2ex\kern2.4ex			\vrule}\hrule}\,}}
\section{Introduction}
Let $\mathfrak g$ be a simple Lie algebra of the type $C_n$ and $\hat{\mathfrak g}=\mathfrak g\otimes\mathbb C[t, t^{-1}]+\mathbb Cc$ the associated affine Lie algebra. 
Let $B$ be a weight basis of $\mathfrak g$ and $\bar B_{<0}=\{b(i)\mid b\in B, i\in\mathbb Z_{<0}\}$, where $b(i)=b\otimes t^i$. Denote by $\mathcal P$ the set of monomials
\begin{equation*}
	\pi=\prod_{b(j)\in \bar B}b(j)^{n_{b(j)}}\in S(\hat{\mathfrak g})
\end{equation*}
and by $\mathcal P_{<0}= \mathcal P\cap S(\hat{\mathfrak g}_{<0})$. Then we say that  $\pi\in\mathcal P_{<0}$ is a colored partition of length $\ell(\pi)$,  degree $|\pi|$ 
and support $\text{supp\,}\pi$,
$$
\quad\qquad\ell(\pi)=\sum_{b(j)\in \bar B}{n_{b(j)}}, \ \ 
|\pi|=\sum_{b(j)\in \bar B}{n_{b(j)}}\cdot j, \ \ 
\text{supp\,}\pi =\{b(j)\in\bar B_{<0}\mid n_{b(j)}>0\},
$$
with colored parts ${b(j)}\in\text{supp\,}\pi$ of degree $j<0$ and color $b\in B$, appearing in the partition $n_{b(j)}$ times.
Instead of the product $\mu\nu\in S(\hat{\mathfrak g}_{<0})$ sometimes we shall write $\mu\cup\nu\in {\mathcal P}_{<0}$ saying that $\mu\cup\nu$ is the partition having all the parts of $\mu$ and $\nu$ together. Likewise we write
$$
\rho\subset\pi\quad\text{if}\quad \pi=\rho\kappa\ (=  \rho\cup\kappa)
$$
for some $\kappa$, saying that $\rho\subset\pi$ is an {\it embedding of $\rho$ into $\pi$}.
 
In \cite{PS1}, a monomial basis of the basic $\hat{\mathfrak g}$-module $L(\Lambda_0)$ is constructed: for each $\pi$ one can choose the corresponding $u(\pi)\in U(\hat{\mathfrak g})$ and construct the PBW-spanning set of the basic module consisting of monomial vectors
\begin{equation}\label{E: PBW spanning}
	u(\pi)v_{\Lambda_0}\in L(\Lambda_0),
\end{equation}
where $v_{\Lambda_0}\in L(\Lambda_0)$ is the highest weight vector.
In order to reduce this spanning set to a basis of $L(\Lambda_0)$, one uses the set of relations $\bar R$ with the property
\begin{equation}\label{E: relations barR}
	L(\Lambda_0)=N(\Lambda_0)/\bar R N(\Lambda_0), \quad \bar R L(\Lambda_0)=\{0\},
\end{equation}
where $N(\Lambda_0)$ is the generalized Verma module with the structure of vertex operator algebra (see \cite{K}, \cite{LL}, \cite{MP1}). Moreover, the relations $r\in \bar R$ are the coefficients of certain vertex operators, which can be written 
(roughly speaking) as infinite combinations of monomials $u(\kappa)$ and can be parametrized as $r=r(\rho)$ by their leading term $\rho\in S(\hat{\mathfrak g}_{<0})$ (i.e. the smallest term $u(\rho)\in U(\hat{\mathfrak g})$ appearing in the sum $r$ with respect to some order on monomials $u(\kappa)$). The set of leading terms of relations $r\in\bar R\,\backslash\{0\}$ is denoted by $ \ell\text{\!\it t\,}(\bar R)\subset \mathcal P_{<0}$.
So for every  $\pi$ that contains as a factor an embedded leading term $\rho\subset\pi$ of a relation $r(\rho)$ we can use the relation (\ref{E: relations barR}) and erase the vector $u(\pi)v_{\Lambda_0}$ from the spanning set (\ref{E: PBW spanning}). The reduced spanning set is a basis; it is parametrized by colored partitions $\pi$ which contain no leading term $\rho$ of relations in $\bar R$. 

In \cite{PS1}, linear independence of the reduced spanning set is proved by constructing a basis of the maximal $\hat{\mathfrak g}$-submodule
\begin{equation}\label{E: maximal submodule}
	N^1(\Lambda_0)=\bar R N(\Lambda_0)\subset N(\Lambda_0).
\end{equation}
Let $u(\rho\subset\pi)=r(\rho)u(\pi/\rho)$. Then the spanning set
\begin{equation}\label{E: spanning of maximal submodule}
	u(\rho\subset\pi)v_{\Lambda_0}\in N^1(\Lambda_0)
\end{equation}
is in \cite{PS1} reduced to a basis by proving the {\it relations among realations}, i.e. by showing that for any two embeddings $\rho_1\subset\pi$ and $\rho_2\subset\pi$ the difference
\begin{equation}\label{E: relations among relations}
	u(\rho_1\subset\pi) - u(\rho_2\subset\pi)
\end{equation}
is an infinite combination of  ``higher in order'' terms $u(\rho\subset\pi')$. The proof is reduced to a combinatorial statement that ``the number of needed relations among relations is equal to the number of constructed relations among relations by using the representation theory''.

We believe that it is possible to construct in this way the monomial bases of \begin{equation}\label{E: relations barR_k}
	L(k\Lambda_0)=N(k\Lambda_0)/\bar R N(k\Lambda_0)
\end{equation} for any positive integer level $k$, and that the main ingredient in the proof would be relations among relations for $\pi$ of lenght $k+2$. In \cite{PS3} this is done for all $n\geq 2$ and $k=2$. In this paper similar combinatorial arguments are used to prove
relations among reletions for $\pi$ of lenght $k+2$ for any $k\geq 2$ but for $n=2$ (see Theorem 5.1  below).

The basic setting and notations are given in the Section 2. Especially,  the
basis $B$ of a simple Lie algebra of $C_2$ type is visualized by the aproppriate triangual scheme. In this paper we are interested in $C_{2}\sp{(1)}$ case, but in some parts it will be convenient to use the notation for arbitrary $C_{n}\sp{(1)}$ (see \cite{PS3}). Therefore, the {\it array of negative root vectors} of $C_{n}\sp{(1)}$  was introduced in the general case for each $n$. In fact, it is a visualization of $\bar B_{<0}$ based on the triangular scheme for $B$. For practical reasons in this paper we use the formalism 
$$\bar B_{<0} \equiv  [\bar{B}_{<0}]_1^{2n+1}\ . $$  
The array of negative root vectors $[\bar{B}_{<0}]_1^{2n+1}$  is just a matrix consisting of $2n+1$ rows and infinitely many columns. At the end of the Section 2 the array $[\bar{B}_{<0}]_1^{5}$ for $C_2^{(1)}$ is presented. 

In Section 3 the leading term $\rho$ of relation $r(\rho)$ is defined for level $k$ standard modules of affine Lie algebra of the type $C_n^{(1)}$. The interpretation of  a leading term as a zig-zag line in $[\bar{B}_{<0}]_1^{2n+1}$ is based on the previous considerations in  \cite{PS3}. The partial order on $[\bar{B}_{<0}]_1^{2n+1}$ necessary for the definition of leading terms is explained on only three successive triangles which form the trapeziod $T$ (see Figure 5). This fact will be of large practical help for  calculations in the rest of the paper. 

The Section 4 is devoted to the procedure for counting of the number of embeddings $\rho\subset\pi$. All calculations are done on three succesive triangles i.e., on the trapozeoid $T$ (see Figure 6). The four types of  $\text{supp\,}\pi $   which allow at least two embedings of leading terms $\rho\subset\pi$ in  the trapozeoid $T$ are classified in Proposition 4.1. Moreover, in this proposition  all $\Sigma_T(\mathcal{X})$ are listed, where $\mathcal{X}$ is one of mentioned support types.  The $\Sigma_T(\mathcal{X})$ is the number of all such  $\text{supp\,}\pi \subset T$ allowing at least two embeddings. Finally, for a colored partition $\pi$ of the length $\ell(\pi)=k+2$ in trapezoid $T\subset [\bar{B}_{<0}]_1^{5}$  we can calculate the corresponding $N_T(\mathcal{X})$ i.e.,  the number related to all colored partions $\pi$ over support of the type $\mathcal{X}$ which allow at least two embeddings.

In Section 5 the main result of this paper, the Theorem 5.1, is proven.  This theorem is the first step in reducing the spanning set 
\begin{equation}\label{E: spanning of maximal submodule_k}
	u(\rho\subset\pi)v_{k\Lambda_0}\in N^1(k\Lambda_0)\ , \ \forall k \in\mathbb Z_{<0}
\end{equation}
to a basis by proving the {\it relations among realations} (\ref{E: relations among relations}). As in \cite{PS3}) the proof of the Theorem 5.1. is also based on the verification of the combinatorial statement that ``the number of needed relations among relations is equal to the number of constructed relations among relations by using the representation theory''. Since both observed numbers depend on the level $k$ it is proven that both of these numbers behave as polynomials of degree $4$ . Therefore, it was sufficient to carry out the proof for five different values of $k$. The natural choice of five values $k$ are $k=1,\dots ,5$. For $k=1,2$ it was done in \cite{PS1} and \cite{PS3}. The case $k=3,4,5$ is the final part of the proof of Theorem 5.1.\\
At the end of Introduction let us mention some of related recent  results with a combinatorial approach \cite{DK}, \cite{CMPP}, \cite{R}, \cite{PT}.

\section{The array of negative root vectors of $C_{2}\sp{(1)}$}

We fix a simple Lie algebra $\mathfrak{g}$ of type $C_2$. For a given Cartan subalgebra $\mathfrak h$ and the corresponding
root system $\Delta$ we can write (as in \cite{B})
\begin{equation*}
	\Delta = \{\pm(\varepsilon_1-\varepsilon_2), \pm(\varepsilon_1+\varepsilon_2), \pm2\varepsilon_1, \pm2\varepsilon_2)\}
\end{equation*}
with simple roots 
$\alpha_1= \varepsilon_1-\varepsilon_2$, $\alpha_2 = 2\varepsilon_2$ and maximal (highest) root $\theta=2 \varepsilon_1$. For each $\alpha\in\Delta$ we choose a root vector $X_{\alpha}$ such that $[X_{\alpha},X_{-\alpha}]=\alpha^{\vee}$. \\
The vectors $\{X_{\alpha}\mid\alpha\in \Delta\}\cup \{\alpha_1^{\vee},\alpha_2^{\vee} \}$ form a basis $B$ of $\mathfrak g$. For root vectors
$X_{\alpha}$ we shall use the following notation:
$$\begin{array}{ccccc}
	X_{\varepsilon_i-\varepsilon_j} \simeq X_{i\underline{j}} &\text{or just}& i\underline{j} & if & i\neq j \\
	X_{\varepsilon_i+\varepsilon_j} \simeq X_{ij} &\text{or just}& {i}j & if & i\leq j\\
	X_{-\varepsilon_i-\varepsilon_j} \simeq X_{\underline{ij}} &\text{or just}& \underline{ij} & if& i\geq j\\
\end{array}
$$
According to this notation we also write $1\underline{1}$ and $2\underline{2}$ for $\alpha_1^{\vee}$ and $\alpha_2^{\vee}$ respectively. The  
 $x_\theta$ is noted by $X_{11}$.
The basis $B$ of $\mathfrak g$ we shall write in a triangular scheme:
\begin{center}
	\begin{tikzpicture} [scale=0.7]
		\node at (0,0) {$11$};\node at (2,0) {$22$};\node at (4,0) {$\underline{22}$};
		\node at (6,0) {$\underline{11}$};
		\node at (1,1) {$12$};\node at (3,1) {$2\underline{2}$};\node at (5,1) {$\underline{21}$};
		\node at (2,2) {$1\underline{2}$};\node at (4,2) {$2\underline{1}$};
		\node at (3,3) {$1\underline{1}$};
	\end{tikzpicture}\ .
\end{center}	
\begin{center}
	Figure 1
\end{center}
However, the usual matrix indexation is a more appropriate notation for the extension of $B$ to   $\bar B_{<0}=\{b(i)\mid b\in B, i\in\mathbb Z_{<0}\}$, where $b(i)=b\otimes t^i$. 
For the above basis $B$ (see Figure 1) we have the following scheme 
\begin{center}
	\begin{tikzpicture} [scale=0.7]
		\node at (0,0) {$11$};\node at (2,0) {$12$};\node at (4,0) {$13$};
		\node at (6,0) {$14$};
		\node at (1,1) {$21$};\node at (3,1) {$22$};\node at (5,1) {$23$};
		\node at (2,2) {$31$};\node at (4,2) {$32$};
		\node at (3,3) {$41$};
	\end{tikzpicture}\ .
\end{center}	
\begin{center}
	Figure 2
\end{center}
As it was mentioned in the Introduction  we are interesed in the $C_{2}\sp{(1)}$ case, but in some parts it will be convenient to use the notation  for arbitrary $C_{n}\sp{(1)}$. Moreover, using the following scheme of an infinite sequence of triangles   
\begin{center}
	\begin{tikzpicture} [scale=0.5]
		\draw (0,0) -- +(10,0) -- +(5,5) -- cycle;
		\node at (5,2) {$B\otimes t^{-1}$};
		\node at (1.3,0.5) {$X_{11}$};\node at (8.7,0.5) {$X_{1,2n}$};\node at (5.5,4) {$X_{2n,1}$};
		\draw (10.5,0.8) -- +(-5,5) -- +(5,5) -- cycle;
		\node at (11,3.5) {$B\otimes t^{-2}$};
		\node at (10.5,1.7) {$X_{2,2n}$};\node at (7.5,5.3) {$X_{2n+1,1}$};\node at (14.5,5.3) {$X_{2n+1,2n}$};
		\draw (11,0) -- +(10,0) -- +(5,5) -- cycle;
		\node at (16,2) {$B\otimes t^{-3}$};
		\node at (13,0.5) {$X_{1,2n+1}$};\node at (19.8,0.5) {$X_{1,4n}$};\node at (17,4) {$X_{2n,2n+1}$};
		\node[fill=black, circle, inner sep=1pt] at (20.5,3){};\node[fill=black, circle, inner sep=1pt] at (21,3){};\node[fill=black, circle, inner sep=1pt] at (21.5,3){};
	\end{tikzpicture}
\end{center}	
\begin{center}
	Figure 3
\end{center}	     
we can extend the matrix indexation of $B$ to  $\bar{B}_{<0}=\coprod_{j>0}{B}\otimes t^{-j}$. The  matrix indexation of $\bar{B}_{<0}$  for arbitrary $C_{n}\sp{(1)}$ we call {\it the array of negative root vectors of $C_{n}\sp{(1)}$}  and denote by $[\bar{B}_{<0}]_1^{2n+1}$.\\
In the Figure 4  the array $[\bar{B}_{<0}]_1^{5}$ of negative root vectors of $C_{2}\sp{(1)}$ is  presented with
\begin{center}
\begin{tikzpicture} [scale=0.7] 
	\node at (0,0) {$11$};
	\node at (1,1) {$21$};
	\node at (2,2) {$31$};
	\node at (3,3) {$41$};
	\node at (4,4) {$51$};
	\node at (2,0) {$12$};
	\node at (3,1) {$22$};
	\node at (4,2) {$32$};
	\node at (5,3) {$42$};
	\node at (6,4) {$52$};
	\node at (4,0) {$13$};
	\node at (5,1) {$23$};
	\node at (6,2) {$33$};
	\node at (7,3) {$43$};
	\node at (8,4) {$53$};
	\node at (6,0) {$14$};
	\node at (7,1) {$24$};
	\node at (8,2) {$34$};
	\node at (9,3) {$44$};
	\node at (10,4) {$54$};
	\node at (8,0) {$15$};
	\node at (9,1) {$25$};
	\node at (10,2) {$35$};
	\node at (11,3) {$45$};
	\node at (10,0) {$16$};
	\node at (11,1) {$26$};
	\node at (12,2) {$36$};
	\node at (12,0) {$17$};
	\node at (13,1) {$27$};
	\node at (14,0) {$18$};
	\node at (15,1) {$28$};
	\node at (14,2) {$37$};
	\node at (16,2) {$38$};
	\node at (13,3) {$46$};
	\node at (15,3) {$47$};
	\node at (17,3) {$48$};
	\node at (12,4) {$55$};
	\node at (14,4) {$56$};
	\node at (16,4) {$57$};
	\node at (18,4) {$58$};
	\draw [dotted] (0,0) -- +(6,0) -- +(3,3) -- cycle;
	\draw [dotted] (4,4) -- +(3,-3) -- +(6,0) -- cycle;
	\draw [dotted] (8,0) -- +(6,0) -- +(3,3) -- cycle;
	\draw [dotted] (12,4) -- +(3,-3) -- +(6,0) -- cycle;
	\node[fill=black, circle, inner sep=1pt] at (17.5,2){};\node[fill=black, circle, inner sep=1pt] at (18,2){};\node[fill=black, circle, inner sep=1pt] at (18.5,2){};
\end{tikzpicture}\ .
\end{center}
\begin{center}
	Figure 4
\end{center}

\section{The leading terms of relations for level $k$ standard modules for $C_2\sp{(1)}$ }
In the Introduction it is explained that the relations $r\in \bar{R}$ are the coefficients of certain vertex operators which can be parametrized as $r=r(\rho)$ by coresponding leading term $\rho\in S(\hat{\mathfrak g})$. The following partial order on $[\bar{B}_{<0}]_1^{2n+1}$ is needed to describe leading terms:

\begin{equation}\label{partial order}
	X_{i,j}\trianglelefteq X_{p,r} \quad \text{if}\quad i\in\{1,\dots , p\}\quad \text{and}\quad j\in\{r,p+r-i\}.	
\end{equation}
\bigskip

\noindent
In other words, $b \trianglelefteq a$ if $b=X_{i,j}$ lies in the cone bellow the vertex $a=X_{p,r}$, as depicted on Figure 5 below:
\begin{center}
	\begin{tikzpicture} [scale=0.5]
		\draw[dashed] (0,0) -- +(10,0) -- +(5,5) -- cycle;
		\draw[dashed] (10.5,0.5) -- +(-5,5) -- +(5,5) -- cycle;
		\draw[dashed] (11,0) -- +(10,0) -- +(5,5) -- cycle;
		\draw (-0.8,-0.3) -- +(22.6,0) --+(16.5,6.2) --+(6.1,6.2) -- cycle;
		
		\node at (9,4) [circle, draw, inner sep=4pt] (A)  {a};
		\node at (10,2) [circle, draw, inner sep=3pt] (B)  {b};
		\draw (5.5,0) -- (A) -- (12.5,0);
		\node at (11.5,3.2) {$b \trianglelefteq a$};
		\node[fill=black, circle, inner sep=1pt] at (20.5,3){};\node[fill=black, circle, inner sep=1pt] at (21,3){};\node[fill=black, circle, inner sep=1pt] at (21.5,3){};
	\end{tikzpicture}\\
	Figure 5
\end{center}
In Figure 5, three successive triangles of $[\bar{B}_{<0}]_1^{2n+1}$ are marked in the trapezoid $T$. As will be seen later, all further considerations and calculations will be made on just such a trapezoid T.\\
With above settings we can define leading terms $\rho$ of relations $r(\rho)\in\bar R$  as follows (for details see \cite{PS2} and \cite{PS3}).
\begin{definition}
	The monomial
	\begin{equation}\label{E:the leading terms of relations}
		\rho=a_1^{m_1}a_2^{m_2}\dots a_s^{m_s}, \quad m_1+m_2+\dots+m_s=k+1,
	\end{equation}
	over {\it downward zig-zag  line} of $s$ points in $[\bar{B}_{<0}]_1^{2n+1}$
	$$
	a_1\vartriangleright a_2\vartriangleright\dots \vartriangleright a_s, \quad 1\leq s\leq k+1 
	$$
	is {\it the leading term}  of the relation $r(\rho)\in\bar R$ for level $k$ standard modules of affine Lie algebra of the type $C_n\sp{(1)}$.
\end{definition}
Moreover, these are all leading terms of $\bar R$ in $\mathcal P_{<0}$  (for more details see Section II. in \cite{PS3}). Also, it is very important to emphasize that the position of the trapezoid $T$ in Figure 5 is in accord with Figure 3 only when the middle triangle is $B\otimes t^j$ for $j$ even, and for $j$ odd the figure should be flipped. However, in our arguments this will make no difference because the flipped zig-zag downward line is again a zig-zag downward line. In Figure 6 the trapezoid $T$ for $[\bar{B}_{<0}]_1^{5}$ is presented where $n=2$.\\
For instance, the following monomial
$$
X_{2\varepsilon_2}^2(-2)\ {\alpha_2^{\vee}}(-2)\ X_{\varepsilon_1-\varepsilon_2}^2(-2)\ X_{\varepsilon_2-\varepsilon_1}(-1)\ X_{-2\varepsilon_1}^3(-1)
$$
where  $m_1+m_2+\dots+m_5= 9 = k+1$ is  the leading term 
$$
\rho = X_{52}^2\ X_{43}\ X_{33}^2\ X_{23}\ X_{14}^3
$$
of the relation $r(\rho)\in\bar R$ for level $8$ standard modules of affine Lie algebra of the type $C_2\sp{(1)}$. The corresponding zig-zag line   
$$
X_{52}\vartriangleright X_{43}\vartriangleright X_{33}\vartriangleright X_{23}\vartriangleright X_{14}\, \simeq\, {52}\vartriangleright {43}\vartriangleright {33}\vartriangleright {23}\vartriangleright {14}\ .
$$
looks like the one in the Figure 6. 
\begin{figure}[h]
\begin{center}
	\begin{tikzpicture} [scale=0.7] 
		\node at (0,0) {$11$};
		\node at (1,1) {$21$};
		\node at (2,2) {$31$};
		\node at (3,3) {$41$};
		\node at (4,4) {$51$};
		\node at (2,0) {$12$};
		\node at (3,1) {$22$};
		\node at (4,2) {$32$};
		\node at (5,3) {$42$};
		\node[fill=green, circle, inner sep=4pt] at (6,4) {};
		\node at (6,4) {$52$};
		\draw[green] (6,4) -- (7,3) -- (6,2) -- (5,1)  -- (6,0);
		\node at (4,0) {$13$};
		\node[fill=green, circle, inner sep=4pt] at (5,1) {};
		\node at (5,1) {$23$};
		\node[fill=green, circle, inner sep=4pt] at (6,2) {};
		\node at (6,2) {$33$};
		\node[fill=green, circle, inner sep=4pt] at (7,3) {};
		\node at (7,3) {$43$};
		\node at (8,4) {$53$};
		\node[fill=green, circle, inner sep=4pt] at (6,0) {};
		\node at (6,0) {$14$};
		\node at (7,1) {$24$};
		\node at (8,2) {$34$};
		\node at (9,3) {$44$};
		\node at (10,4) {$54$};
		\node at (8,0) {$15$};
		\node at (9,1) {$25$};
		\node at (10,2) {$35$};
		\node at (11,3) {$45$};
		\node at (10,0) {$16$};
		\node at (11,1) {$26$};
		\node at (12,2) {$36$};
		\node at (12,0) {$17$};
		\node at (13,1) {$27$};
		\node at (14,0) {$18$};
		\draw [dotted] (0,0) -- +(6,0) -- +(3,3) -- cycle;
		\draw [dotted] (4,4) -- +(3,-3) -- +(6,0) -- cycle;
		\draw [dotted] (8,0) -- +(6,0) -- +(3,3) -- cycle;
		\draw  (-0.8,-0.3) -- (3.7,4.4) -- (10.3,4.4) -- (14.8,-0.3) -- cycle;
	\end{tikzpicture}
\end{center}
Figure 6
\end{figure}

\section{Colored partitions allowing at least two embeddings of leading terms }

In this section we count the number of embeddings $\rho\subset\pi$  in three successive triangles of $[\bar{B}_{<0}]_1^{2n+1}$ (i.e., in the trapezoid $T$ on Figure 5). 
Following the notation from previous sections, we can write the colored partition $\pi\in\mathcal P_{<0}$ as
\begin{equation}\label{E:colored partition}
	\pi=\prod\sb{a\in[\bar{B}_{<0}]_1^{2n+1}} a^{\pi(a)}\ 
\end{equation}
where  $\ell(\pi)=k+2$, $\text{supp\,}\pi =\{a_1,a_2,\dots a_s\}$ and $|\text{supp\,}\pi| = s\leq 2n+1$.\\
Now let us assume that $\pi$ allows two embeddings of leading terms of relations for level $k$ standard modules i.e. $\rho_1\subset \pi$ and $\rho_2\subset \pi$ are two different zig-zag lines of length $k+1$.\\  
Let us summarize the descriptions of all possible cases for $\text{supp\,}\pi$. In \cite{PS3} all such $\text{supp\,}\pi$ are classfied in Lemma 3.1. i.e. $\text{supp\,}\pi $ is one of the following types 
	\begin{itemize}
		\item[$(A_{s})$] $\text{supp\,}\pi =\{a_1, \dots, a_s\}$, $s\geq 2$, $a_1\vartriangleright\dots\vartriangleright a_s$.
		\item[$(B_{s\,\delta})$] $\text{supp\,}\pi =\{a_1, \dots, a_{s},b,c\}$, $s\geq 1$, $a_1\vartriangleright\dots\vartriangleright a_{s}$, $a_{s}\vartriangleright b$, $a_{s}\vartriangleright c$ and $b$ and $c$ are not comparable. We set $\delta$ to be $\vert$  if $b$ and $c$ are in the same row, and $\vert\vert $ otherwise.
		\item[$(C_{\delta\, s})$] $\text{supp\,}\pi =\{b,c,a_1, \dots, a_{s}\}$, $r\geq 1$, $a_1\vartriangleright\dots\vartriangleright a_{s}$, $b\vartriangleright a_1$, $c\vartriangleright a_1$ and $b$ and $c$ are not comparable. We set $\delta$ to be $\vert$  if $b$ and $c$ are in the same row, and $\vert\vert $ otherwise
		\item[$(D_{s\, \delta\, t})$] $\text{supp\,}\pi =\{a_1, \dots, a_s,b,c,d_1,\dots,d_t\}$,  $s,t\geq 1$, $a_1\vartriangleright\dots\vartriangleright a_s$, $a_s\vartriangleright b\vartriangleright d_1$, $a_s\vartriangleright c\vartriangleright d_1$, $d_1\vartriangleright\dots\vartriangleright d_t$, and $b$ and $c$ are not comparable. We set $\delta$ to be $\vert$  if $b$ and $c$ are in the same row, and $\vert\vert $ otherwise.
	\end{itemize}
	Note that in the case $(A_{s})$ it holds $|\text{supp\,}\pi| = s$, in cases $(B_{s\,\delta})$ and $(C_{\delta\, s})$ $|\text{supp\,}\pi| = s+2$ and in the case $(D_{s\, \delta\, t})$ $|\text{supp\,}\pi| = s+t+2$.\\
	In an effort to count the “number of relations among relations needed” for embeddings $\rho\subset\pi$ on the trapezoid $T$ we need the following notation
	\begin{eqnarray*}
	N(\pi)=\max\{\#\mathcal E(\pi)-1,0\} &\text{where}& \mathcal E(\pi)=\{\rho\in \ell\text{\!\it t\,}(\bar R)\mid \rho\subset\pi\}
	\end{eqnarray*}
and
	\begin{eqnarray*}
	\Sigma_T(\mathcal{X}) & = & \sum_{S\subset T,\ S\text{ is of the type }\mathcal{X}}1\\
	&&\nonumber\\
	N_T(\mathcal{X}) &=& \sum_{\pi, \, \text{supp\,}\pi\subset T, \,\text{supp\,}\pi\text{ is of the type }\mathcal{X}}N(\pi)
	\end{eqnarray*}
where  $\mathcal{X}$ is one of  above types for $\text{supp\,}\pi $. 
Due to Lemmas 3.2. and 3.3, again from \cite{PS3}, we have: 
\begin{eqnarray}\label{N_{T}(X)}
	N_{T}(A_s) & = & (s-1){k+1\choose s-1}\,\Sigma_T(A_s)\nonumber\\
	&&\nonumber\\
	N_{T}(B_{s\,\delta}) & = & {k-1\choose s-1}\,\Sigma_{T}(B_{s\,\delta})\nonumber\\
	&&\\
	N_{T}(C_{\delta\, s}) & = & {k-1\choose s-1}\,\Sigma_{T}(C_{\delta\, s})\nonumber\\
	&&\nonumber\\
	N_{T}(D_{s \,\delta\, t}) & = & {k-1\choose s+t-1}\,\Sigma_{T}(D_{s \,\delta\, t})\ .\nonumber
\end{eqnarray}
Moreover, in  \cite{PS3}  $N_T(\mathcal{X})$ are explicitly calculated for $[\bar{B}_{<0}]_1^{2n+1}$, i.e., for arbitrary $n$. Furthermore, in the case  $[\bar{B}_{<0}]_1^5$ where $n=2$ we can state the following proposition.
\begin{proposition}\label{L: klasifikacija dva ulaganja}
 	Let $\ell(\pi)=k+2$ and let the trapezoid $T\subset [\bar{B}_{<0}]_1^{5}$ be oriented as in Figure 6. Assume that $\pi$ allows two embeddings of leading terms of relations for level $k$ standard modules. Then $\text{supp\,}\pi $ is one of the following types :
$$
\begin{array}{rcll}
 		(A_{s}) &for& s=2,3,4,5 & \nonumber\\
 		(B_{s\,|}) &for& s=1,2,3,4& \nonumber\\
 		(B_{s\,||}) &for& s=1,2,3& \nonumber\\
 		(C_{|\, s}) &for& s=1,2,3,4& \nonumber\\
 		(C_{||\, s}) &for& s=1,2,3&  \nonumber\\
 		(D_{s\, |\, t}) &for& s,t=1,2,3& where\ s+t<5 \nonumber\\
 		(D_{s\, ||\, t}) &for& s,t=1,2 &where\ s+t<4 .\nonumber
\end{array}
$$
For the above listed types, respectively we have
\begin{equation}\label{Sigma_T}
	\begin{array}{cccc}
	\Sigma_T(A_5) =  64 & \Sigma_T(A_4)  =  216& \Sigma_T(A_3) =  268&	\Sigma_T(A_2) =  145\\
	&&&\\
	\Sigma_{T}(B_{4\,|}) =  32& \Sigma_{T}(B_{3\,|}) =  132& \Sigma_{T}(B_{2\,|}) =  211& \Sigma_{T}(B_{1\,|}) =  161\\
	\Sigma_{T}(B_{3\,||}) = 32& \Sigma_{T}(B_{2\,||}) = 124& \Sigma_{T}(B_{1\,||}) = 182 &\\
	&&&\\
	\Sigma_{T}(C_{4\,|})=24&\Sigma_{T}(C_{3\,|})=96&\Sigma_{T}(C_{2\,|})=147&\Sigma_{T}(C_{1\,|})=105\\
	\Sigma_{T}(C_{3\,||})=24&\Sigma_{T}(C_{2\,||}) =90&\Sigma_{T}(C_{1\,||})=126&\\
	&&&\\
	\Sigma_{T}(D_{2|2})=16&\Sigma_{T}(D_{3|1})=16&\Sigma_{T}(D_{1|3})= 16&	\\
	\Sigma_{T}(D_{2|1})=62&\Sigma_{T}(D_{1|2})=62&\Sigma_{T}(D_{1|1}) = 95&	\\
	\Sigma_{T}(D_{2||1})=16&\Sigma_{T}(D_{1||2})=16&\Sigma_{T}(D_{1||1})= 58 & \ .
\end{array}
\end{equation}
 \end{proposition} 
\begin{proof}
Since the array $[\bar{B}_{<0}]_1^{5}$ consists of $5$ rows, it is obvious that $\pi$ cannot be composed of more than five elements of the base $B$ i.e. $|\text{supp}\, \pi|\leq 5$. The  possible types of $\text{supp\,}\pi $  are listed above (see also Lemma 3.1 in \cite{PS3}) and the first statement of the proposition is true. The second statement follows from the fact that $n=2$. More precisely, Lemma 3.4 implies statement for type $(A_s)$, Lemmas 3.6 and 3.10 for the type $(B_{s\, \delta})$, Lemmas 3.7 and 3.11 for the type $(C_{\delta\, s})$ and finally Lemmas 3.8 and 3.12 for the type $(D_{s\, \delta\, t})$. All mentioned lemmas are cited from \cite{PS3}.  	
\end{proof}	

\section{Level $k$ relations for $C_2^{(1)}$ for two embeddings in $\pi$ of length $\ell(\pi)=k+2$}
In this sectionr, we will present the main result of this paper, which is Theorem 5.1. Before stating the theorem, we will introduce additional notions.\\
For $\ell\geq0$ and $m\in\mathbb Z$ set $$\mathcal P^\ell =\{\pi\in\mathcal P\mid \ell(\pi)=\ell\}\quad \text{and}\quad\mathcal P^\ell(m) =\{\pi\in\mathcal P\mid \ell(\pi)=\ell, |\pi|=m\}\ .$$  
For a colored partition (\ref{E:colored partition}) we define {\it the shape of $\pi$} as the ordinary partition of  $|\pi|\in-\mathbb N$
$$
\text{sh\,}\pi=\prod\sb{a\in [\bar{B}_{<0}]_1^{2n+1}} |a|^{\pi(a)}
$$
with parts $|a|\in-\mathbb N$ for $a\in\text{supp\,}\pi=\{a\in[\bar{B}_{<0}]_1^{2n+1}\mid\pi(a)>0\}$. Note that $$|\text{sh\,}\pi|=|\pi|=\sum\sb{a\in[\bar{B}_{<0}]_1^{2n+1}} |a|\cdot{\pi(a)}\ .$$
As usual the ordinary partitions are presented by  Young tableauxes. For instance, the Young tableaux for zig-zag line 
$$
\rho = X_{52}\ X_{43}\ X_{33}\ X_{23}\ X_{14}$$
$$ \simeq X_{2\varepsilon_2}(-2)\ {\alpha_2^{\vee}}(-2)\ X_{\varepsilon_1-\varepsilon_2}(-2)\ X_{\varepsilon_2-\varepsilon_1}(-1)\ X_{-2\varepsilon_1}(-1)
$$
interpreted as the ordinary partition looks like
$$	\begin{array}{l}
	\sq\sq \\
	\sq\sq\\
	\sq\sq \\
	\sq  \\
	\sq        
\end{array}    		
$$
Indeed, based on Figure 6, it follows that the monomial $X_{52}\ X_{43}\ X_{33}$ is contained in the second triangle, and the monomial $X_{23}\ X_{14}$ in the first triangle.\\
Finally, the main result of this paper is the following theorem:
\begin{theorem}\label{T:the main theorem}
	For any two embeddings
	$\rho_1 \subset \pi$ and $\rho_2 \subset \pi$ in $\pi\in\mathcal P^{k+2}(m)$,
	where $\rho_1, \rho_2 \in\ell \!\text{{\it t\,}}(\bar{R})$, we
	have a level $k$ relation for $C_2^{(1)}$
	\begin{equation}\label{E:9.2}
		u(\rho_1 \subset \pi) - u(\rho_2 \subset \pi) 
		=\sum_{\pi\prec \pi', \ \rho\subset\pi'}c_{\rho\subset\pi'}\,u(\rho\subset\pi').
	\end{equation}
\end{theorem}
\begin{proof}
By using methods in  \cite{PS3} to prove this theorem,  we need to count ''the number of  two-embeddings for $\text{sh\,} \pi=p$''. More precisely for a fixed ordinary partition $p$ of length $k+2$ we need to count 
$$
\sum_{\text{sh\,} \pi=p, \, \pi\in\mathcal P^{k+2}(m)} N(\pi).  
$$
By Theorem 7.4 \cite{PS1} for  $\pi\in\mathcal P^{k+2}(m)$ we have a space of relations for annihilating fields 
\begin{equation}
	Q_{k+2}(m)  =  U(\mathfrak g)q_{(k+1)\theta}(m)\oplus U(\mathfrak g)q_{(k+2)\theta}(m)\oplus U(\mathfrak g)q_{(k+2)\theta-\alpha\sp*}(m)
\end{equation}
where $\alpha\sp*=\alpha_1=\varepsilon_1 -\varepsilon_2$. Then we have
\begin{equation}\label{E:9.1}
	\sum_{\pi\in\mathcal P^{k+2}(m)} N(\pi)\geq\dim Q_{k+2}(m)\ .
\end{equation}
By Theorem 9.2 in \cite{PS1}, in order to verify this theorem  it is enough to prove that in (\ref{E:9.1}) the equality holds.
However it is  much easier to count the number $N(\pi)$ for all $\pi$ with support in a trapezoid $T$ of three successive triangles (see Figure 6), i.e. 
\begin{equation}\label{E:9.1T}
N_T(k)=\sum_{\pi, \, \ell(\pi)=k+2, \, \text{supp\,} \pi\subset T} N(\pi) 
\end{equation}
where $k$ is arbitrary level. It is important to emphasize that the right side of the equation (\ref{E:9.1}) also changed accordingly.\\
Of course, in the case of $n=2$, we have a significantly simplified situation considering that the parameter $|\text{supp}\, \pi|$ is limited by $5$ (see the first statement of Proposition 4.1.). From equations (\ref{N_{T}(X)}) and (\ref{Sigma_T}) we have the following list of all possible  $N_{T}(\mathcal{X})$ for $n=2$ i.e. for $\tilde{\mathfrak{g}}=\tilde{\mathfrak{sp}}_{4}$
$$	\begin{array}{cccc}
		N_{T}(A_5) =  256 \binom{k+1}{4} & N_{T}(A_4) =  648\binom{k+1}{3}& 
		N_{T}(A_3) =  536\binom{k+1}{2}&	N_{T}(A_2) =  145\binom{k+1}{1}\nonumber\\
		&&&\nonumber\\
		N_{T}(B_{4\,|})=  32\binom{k-1}{3}& N_{T}(B_{3\,|}) =  132\binom{k-1}{2}& N_{T}(B_{2\,|}) =  211\binom{k-1}{1}& N_{T}(B_{1\,|}) =  161\binom{k-1}{0}\nonumber\\
		N_{T}(B_{3\,||}) = 32\binom{k-1}{2}& N_{T}(B_{2\,||}) = 124\binom{k-1}{1}& N_{T}(B_{1\,||}) = 182\binom{k-1}{0} &\nonumber\\
		&&&\nonumber\\
		N_{T}(C_{4\,|})=24\binom{k-1}{3}&N_{T}(C_{3\,|})=96\binom{k-1}{2}&N_{T}(C_{2\,|})=147\binom{k-1}{1}&N_{T}(C_{1\,|})=105\binom{k-1}{0}\nonumber\\
		N_{T}(C_{3\,||})=24\binom{k-1}{2}&N_{T}(C_{2\,||}) =90\binom{k-1}{1}&N_{T}(C_{1\,||})=126\binom{k-1}{0}&\nonumber\\
		&&&\nonumber\\
		N_{T}(D_{2|2})=16\binom{k-1}{3}&N_{T}(D_{3|1})=16\binom{k-1}{3}&N_{T}(D_{1|3})= 16\binom{k-1}{3}&	\nonumber\\
		N_{T}(D_{2|1})=62\binom{k-1}{2}&N_{T}(D_{1|2})=62\binom{k-1}{2}&N_{T}(D_{1|1}) = 95\binom{k-1}{1}&	\nonumber\\
		N_{T}(D_{2||1})=16\binom{k-1}{2}&N_{T}(D_{1||2})=16\binom{k-1}{2}&N_{T}(D_{1||1})= 58\binom{k-1}{1} & .\nonumber
	\end{array}
$$
From above list it is obvious  that $N_T(k)$ in (\ref{E:9.1T}) is a polynomial of degree $4$. \\
On the other side, the Weyl dimension formula in the case of symplectic simple Lie algebra $\mathfrak{g}=\mathfrak{sp}_{4}$ gives
\begin{eqnarray}
	\dim L((k+1)\theta) &=& {2k+5\choose 3},\nonumber\\
	\dim L((k+2)\theta) &=& {2k+7\choose 3},\label{E:11.1}\\
	\dim L((k+2)\theta-\alpha^{\star}) &=& \frac{(2k+7)(2k+5)(2k+3)}{3}\nonumber
\end{eqnarray}
From (\ref{E:11.1}) we have
\begin{equation}\label{dim Q_{n=2}}
	\dim Q_{k+2}(m) = 4\binom{2k+6}{3}
\end{equation}
i.e. the $\dim Q_{k+2}(m)$ is a a polynomial of degree $3$. 
For the trapezoid scheme $T$ of three successive triangles all shapes $\text{sh\,}\pi$ that appear with two embeddings of leading terms are listed by following $2k+6$ Young tableauxes 
$$	
\begin{array}{clllcllllc}
		-m=&&&&-m=&-m=&&&&-m=\\
		(k+2)&&&&2(k+2)&2(k+2)&&&&3(k+2)\\
		\sq &\sq\sq&\cdots&\sq\sq&\sq\sq&\sq\sq\sq&\sq\sq\sq&\cdots&\sq\sq\sq&\sq\sq\sq \\
		\sq &\sq &\cdots&\sq\sq&\sq\sq&\sq\sq&\sq\sq&\cdots&\sq\sq\sq&\sq\sq\sq \\
		\vdots&\vdots&\vdots&\vdots&\vdots&\vdots&\vdots&\cdots&\vdots&\vdots\\
		\sq &\sq &\cdots&\sq\sq&\sq\sq&\sq\sq&\sq\sq&\cdots&\sq\sq\sq&\sq\sq\sq \\
		\sq &\sq &\cdots&\sq\sq &\sq\sq&\sq\sq&\sq\sq&\cdots&\sq\sq\sq&\sq\sq\sq \\ 
		\sq &\sq &\cdots&\sq &\sq\sq&\sq&\sq\sq&\cdots&\sq\sq&\sq\sq\sq     
\end{array} \ .   		
$$
\begin{center}
	\text{Figure 8}
\end{center}
Note that only $-m=2(k+2)$ has two corresponding Young tableauxes.
Hence the right hand side of inequality (\ref{E:9.1}) can be replaced by 
\begin{equation}\label{sum P^{k+2(n)}}
	 (2k+5)\times 	\dim Q_{k+2}(m)- 2\times\dim L((k+2)\theta)\ .
\end{equation}
Since $\dim Q_{k+2}(m)$  are  polynomials of degree $3$ (see \ref{dim Q_{n=2}}) it follows  that
the expression (\ref{sum P^{k+2(n)}}) is also a polynomial of maximum degree $4$.
Moreover from (\ref{E:9.1}) it follows that 
\begin{equation}\label{star}
	N_T(k)\geq (2k+5)\times 	\dim Q_{k+2}(m)- 2\times\dim L((k+2)\theta)
\end{equation}	
It is important to point out, if equality holds in (\ref{star}) then equality must also hold in (\ref{E:9.1}).
Since both sides of the inequality (\ref{star}) are polynomials of of maximum degree $4$ it is sufficient to check that the equality holds in (\ref{star})  for  five different values of $k$.

In \cite{PS1} the theorem is proved for basic modules i.e. in the case $k=1$. Moroever in the \cite{PS3}	the above theorem statement is proved for $k=2$. It is important to emphasize that both of the above results are proved for arbitrary $n$, i.e. for every symplectic affine Lie algebra $C_n^{(1)}$.\\
Finally, to finish the proof of theorem it is enough to check equality in (\ref{star})  for the  trapezoid scheme case  where $k=3,4,5$. Respectively, we have the following:
\begin{itemize}
	\item[$k=3$] 
	$$
	\begin{array}{cccc}
	N_{T}(A_5) =  256  & N_{T}(A_4) =  2592& 
	N_{T}(A_3) =  3212&	N_{T}(A_2) =  580\nonumber\\
	&&&\nonumber\\
	N_{T}(B_{3\,|}) =132& N_{T}(B_{2\,|}) =  422 & N_{T}(B_{1\,|}) =  161\nonumber&\\
	N_{T}(B_{3\,||}) = 32& N_{T}(B_{2\,||}) = 248& N_{T}(B_{1\,||}) = 182 &\nonumber\\
	&&&\nonumber\\
	N_{T}(C_{3\,|})=96&N_{T}(C_{2\,|})=294&N_{T}(C_{1\,|})=105&\nonumber\\
	N_{T}(C_{3\,||})=24&N_{T}(C_{2\,||}) =180&N_{T}(C_{1\,||})=126&\nonumber\\
	&&&\nonumber\\
	N_{T}(D_{2|1})=62&N_{T}(D_{1|2})=62&N_{T}(D_{1|1}) = 190&	\nonumber\\
	N_{T}(D_{2||1})=16&N_{T}(D_{1||2})=16&N_{T}(D_{1||1})= 116 & \nonumber
	\end{array}
	$$
	\item[$k=4$]
	$$
	\begin{array}{cccc}
	N_{T}(A_5) =  1280 & N_{T}(A_4) =  6480& 
	N_{T}(A_3) =  5360&	N_{T}(A_2) =  725\nonumber\\
	&&&\nonumber\\
	N_{T}(B_{4\,|})=  32& N_{T}(B_{3\,|}) =  396& N_{T}(B_{2\,|}) =  633& N_{T}(B_{1\,|}) =  161\nonumber\\
	N_{T}(B_{3\,||}) = 96& N_{T}(B_{2\,||}) = 372& N_{T}(B_{1\,||}) = 182 &\nonumber\\
	&&&\nonumber\\
	N_{T}(C_{4\,|})=24&N_{T}(C_{3\,|})=288&N_{T}(C_{2\,|})=441&N_{T}(C_{1\,|})=105\nonumber\\
	N_{T}(C_{3\,||})=72&N_{T}(C_{2\,||}) =270& N_{T}(C_{1\,||})=126&\nonumber\\
	&&&\nonumber\\
	N_{T}(D_{2|2})=16&N_{T}(D_{3|1})=16&N_{T}(D_{1|3})= 16&	\nonumber\\
	N_{T}(D_{2|1})=186&N_{T}(D_{1|2})=186&N_{T}(D_{1|1}) = 285&	\nonumber\\
	N_{T}(D_{2||1})=48&N_{T}(D_{1||2})=48&N_{T}(D_{1||1})= 174 & \nonumber
	\end{array}
	$$
	\item[$k=5$]
	$$
	\begin{array}{cccc}
		N_{T}(A_5) =  3840 & N_{T}(A_4) =  12960& 
		N_{T}(A_3) =  8040&	N_{T}(A_2) =  870\nonumber\\
		&&&\nonumber\\
		N_{T}(B_{4\,|})=  128& N_{T}(B_{3\,|}) =  792& N_{T}(B_{2\,|}) =  844& N_{T}(B_{1\,|}) =  161\nonumber\\
		N_{T}(B_{3\,||}) = 192& N_{T}(B_{2\,||}) = 496& N_{T}(B_{1\,||}) = 182 &\nonumber\\
		&&&\nonumber\\
		N_{T}(C_{4\,|})=96&N_{T}(C_{3\,|})=576&N_{T}(C_{2\,|})=588&N_{T}(C_{1\,|})=105\nonumber\\
		N_{T}(C_{3\,||})=144&N_{T}(C_{2\,||}) =360& N_{T}(C_{1\,||})=126&\nonumber\\
		&&&\nonumber\\
		N_{T}(D_{2|2})=64&N_{T}(D_{3|1})=64&N_{T}(D_{1|3})= 64&	\nonumber\\
		N_{T}(D_{2|1})=372&N_{T}(D_{1|2})=372&N_{T}(D_{1|1}) =380&	\nonumber\\
		N_{T}(D_{2||1})=96&N_{T}(D_{1||2})=96&N_{T}(D_{1||1})= 232 & \nonumber
	\end{array}
	$$ 
\end{itemize} 
Finally, after the summation of all $N_{T}(\mathcal{X})$  for $k=3,4,5$ follows sequentially
\begin{eqnarray}\label{N_{T}(3,4,5)}
	N_{T}(3) & = & 9108\nonumber\\
	N_{T}(4) & = & 18018\\
	N_{T}(5) & = & 32240\nonumber 
\end{eqnarray}

For the right hand side of the inequality (\ref{E:9.1}) from (\ref{dim Q_{n=2}}) we have
\begin{eqnarray}\label{tri sume}
	\dim Q_5(m) &=& \dim L(4\theta) + \dim L(5\theta)  + \dim L(5\theta-\alpha^{\star}) =880\nonumber\\ 
	\dim Q_6(m) &=& \dim L(5\theta) + \dim L(6\theta)  + \dim L(6\theta-\alpha^{\star}) = 1456\\ 
	\dim Q_7(m) &=& \dim L(6\theta) + \dim L(7\theta)  + \dim L(7\theta-\alpha^{\star}) = 2240\nonumber
\end{eqnarray}
Furthermore,  all shapes in $T$ which appear with two embeddings of leading terms are listed by Young tableauxes for cases  $k=3,\, 4,\, 5$ like in Figure 8.\\
From   (\ref{E:11.1}), (\ref{sum P^{k+2(n)}}), (\ref{tri sume}) and Figure 8 follows   
\begin{eqnarray}\label{sum P^{k+2(345)}}
	\sum_{m=5}^{15}\sum_{\pi\in\mathcal P^5(m)} N(\pi) &  = &11\times 880- 2\times\binom{13}{10}= 9108\nonumber\\
	\sum_{m=6}^{18}\sum_{\pi\in\mathcal P^6(m)} N(\pi) &  = &13\times 1456- 2\times\binom{15}{12}= 18018\\
	\sum_{m=7}^{21}\sum_{\pi\in\mathcal P^7(m)} N(\pi) &  = &15\times 2240- 2\times\binom{17}{14}= 32240\ .\nonumber
\end{eqnarray}
Now, from (\ref{N_{T}(3,4,5)}) and (\ref{sum P^{k+2(345)}}) it is obvious that equality holds in (\ref{star}) and in  (\ref{E:9.1})  i.e. the theorem is proven.
\end{proof}


\section*{Acknowledgement}
This work was supported by the project "Implementation of cutting-edge research
and its application as part of the Scientific Center of Excellence for Quantum
and Complex Systems, and Representations of Lie Algebras", PK.1.1.02, European
Union, European Regional Development Fund.\\
Also, this work has been supported in part by Croatian Science Foundation under the project IP-2022-10-9006.

The author of this paper is grateful to Mirko Primc for support and many stimulating discussions.


\end{document}